\documentclass[11pt]{article}

\usepackage{amsmath,amssymb,amsthm}
\usepackage{enumitem}
\usepackage{geometry}
\usepackage{mathrsfs}
\newtheorem{theorem}{Theorem}[section]
\newtheorem{lemma}[theorem]{Lemma}
\newtheorem{proposition}[theorem]{Proposition}
\newtheorem{corollary}[theorem]{Corollary}
\theoremstyle{definition}
\newtheorem{definition}[theorem]{Definition}
\newtheorem{remark}[theorem]{Remark}
\newtheorem{example}[theorem]{Example}

\newcommand{\Fun}{\mathrm{Fun}}
\newcommand{\rad}{\mathrm{rad}}
\newcommand{\Def}{\mathsf{Def}}
\newcommand{\End}{\mathrm{End}}
\newcommand{\C}{\mathbb{C}}

\newcommand{\F}{\mathbb{F}}
\newtheorem{principle}[theorem]{Principle}

\title{Algebraic Phase Theory I: Radical Phase Geometry and Structural Boundaries}
\author{Joe Gildea}
\author{
Joe Gildea\\
Department of Computing Science and Mathematics,\\
School of Informatics and Creative Arts,\\
Dundalk Institute of Technology\\
\texttt{gildeajoe@gmail.com}}
\date{}

\begin{document}
\maketitle

\begin{abstract}
We develop Algebraic Phase Theory (APT), an axiomatic framework for extracting
intrinsic algebraic structure from phase based analytic data.  From minimal
admissible phase input we prove a general phase extraction theorem that yields
algebraic Phases equipped with functorial defect invariants and a uniquely
determined canonical filtration.  Finite termination of this filtration forces
a structural boundary: any extension compatible with defect control creates new
complexity strata.

These mechanisms are verified in the minimal nontrivial setting of quadratic
phase multiplication operators over finite rings with nontrivial Jacobson
radical.  In this case nilpotent interactions produce a finite filtration of
quadratic depth, and no higher degree extension is compatible with the axioms.
This identifies the radical quadratic Phase as the minimal example in which
defect, filtration, and boundary phenomena occur intrinsically.
\end{abstract}



\section{Introduction}

Many analytic systems exhibit structured behaviour that is governed not by
smoothness or regularity, but by the interaction of phases. Over fields, such
interactions are algebraic in nature: additive derivatives measure polynomial
structure, polarization encodes bilinear and quadratic interactions, and
representation theory organizes the resulting operator behaviour through models
such as the Weil representation \cite{Weil1964SurCertains}, the Heisenberg and
related nilpotent constructions of Howe \cite{Howe1979Heisenberg}, and the
induced representation theory of Mackey \cite{Mackey1958Induced}.

Over finite rings with nilpotent coefficients, these principles break down.
Phase interactions become sensitive to behaviour that is invisible to ordinary
characters, and naive extensions of bilinear or quadratic theory fail. These
failures are structural in origin. They reflect genuine algebraic obstructions
created by nilpotence rather than deficiencies of analytic formulation.

The first central observation of this paper is that, in such settings, analytic
phase phenomena admit a finite algebraic description only up to a sharp
boundary. Beyond this boundary, no extension of the theory can preserve
functoriality, intrinsic meaning, and finite control. The purpose of this work
is to isolate and explain this boundary.

\medskip

To accomplish this, we introduce \emph{Algebraic Phase Theory} (APT). APT is a
framework that extracts canonical algebraic structure directly from analytic
phase behaviour. Starting from analytic phase data, consisting of a family of
functions together with additive derivatives and basic interaction laws, APT
identifies three intrinsic invariants: the defect, which records the first
nonvanishing additive derivative; defect tensors, which encode the first
obstruction to additivity; and a canonical finite filtration whose layers are
determined entirely by defect degree.

From these analytic invariants, APT constructs a canonical algebraic object
called the \emph{algebraic Phase}. This object captures the interaction laws of
the original analytic system in a functorial and presentation-independent way.
It retains all defect information, carries a finite filtration determined by
analytic degree, and terminates at a structural boundary that cannot be crossed
without violating the axioms. The translation from analytic to algebraic
structure is functorial, intrinsic, and independent of analytic scaffolding
such as norms, topology, or operator theory.

\medskip

The perspective of APT differs from existing approaches based on harmonic
analysis, nilmanifolds, polynomial dynamics, or representation-theoretic
models. Such frameworks provide powerful realizations of phase interactions but
rely on analytic or geometric structure. APT instead works directly with
additive derivatives, identifies defect as the central invariant, and constructs
the algebraic Phase purely from functorial analytic inputs.

This conceptual shift clarifies the relationship between analytic and algebraic
viewpoints. Analytic and algebraic behaviour coincide through defect. Higher
order structure is organized by a canonical finite filtration. The depth of this
filtration is forced entirely by analytic degree. Structural boundaries such as
quadratic or cubic limits appear as intrinsic features rather than artifacts of
representation. Rigidity phenomena in ergodic theory, additive combinatorics,
harmonic analysis, and algebraic coding theory share underlying algebraic
features even though they arise in very different analytic settings. APT
isolates exactly this algebraic core and places it within a unified framework.

\medskip

APT distinguishes between strongly admissible and weakly admissible phase data.
In the strongly admissible regime, defect propagation is fully controlled by the
axioms and all boundaries are dictated by defect rank. In weakly admissible
settings, defect and filtration still exist but higher order propagation may
fail. The present paper works entirely in the strongly admissible regime. Later
papers in this series will show that weak admissibility leads to additional
structural collapse phenomena not visible at the level of defect.

\medskip

The first nontrivial example of APT arises from quadratic phases over the
radical ring \( R = \mathbb{F}_2[u] / (u^2) \). In this setting, defect appears
precisely at second additive order, all third order additive derivatives vanish,
the algebraic Phase has a canonical finite filtration of depth two, and no
further extension is possible without violating the axioms. This model plays a
foundational role. It is the simplest example in which defect, filtration,
functoriality, finite termination, and structural boundaries all occur
intrinsically. The axioms of APT are therefore not imposed externally but are
forced by the behaviour already visible in this model.

\medskip

Analytic phase phenomena are often organized through higher order structures
such as nilmanifolds in ergodic theory and nilspaces in additive combinatorics,
arising through inverse theorems for uniformity norms
\cite{Gowers2001Fourier,HostKra2005Nilmanifolds,GreenTaoZiegler2012Inverse,
CamarenaSzegedy2010Nilspaces}. Related algebraic structures appear in the theory
of polynomial and nilpotent actions \cite{BergelsonLeibman1996Polynomial}. APT
does not replace these theories. Instead, it isolates the intrinsic algebraic
content common to them, independent of analytic or geometric realization. From
this point of view, objects such as nilmanifolds or representation categories may
be regarded as realizations of underlying algebraic phase structures rather than
as fundamental objects in their own right.

\medskip

The main results of this paper establish Algebraic Phase Theory in the strongly
admissible regime. We show that defect data canonically induce a finite
filtration for any strongly admissible algebraic Phase, and that finite
termination forces a sharp structural boundary. In the radical quadratic model,
this boundary occurs at quadratic depth and cannot be extended.

\medskip

This work is structural rather than classificatory. We do not attempt to
classify all algebraic Phases arising from analytic behaviour or to develop
their representation theory. Instead, we isolate the minimal axiomatic
conditions under which phase interactions admit a coherent, functorial, and
finitely controlled algebraic realization, and we identify the obstructions that
arise beyond this regime.

\medskip

This paper is the first in a six part series developing Algebraic Phase Theory.
The subsequent papers \cite{GildeaAPT2,GildeaAPT3,GildeaAPT4,GildeaAPT5,GildeaAPT6}
treat weak admissibility, categorical and equivalence phenomena, boundary
calculus, and intrinsic reconstruction. The present work provides the strongly
admissible foundation upon which the remainder of the series is built.

\section{Basic Notions for Phase Extraction}

This section isolates the minimal pre-algebraic input from which Algebraic Phase
Theory operates. Rather than beginning with an algebraic structure, we specify
only the behavioural data required to describe phase interaction, prior to any
notion of defect, filtration, or completion.

The input data considered here is intentionally sparse. It encodes phase-like
behaviour and interaction but carries no a priori defect stratification,
canonical filtration, or algebraic completeness. All such structure is extracted
intrinsically at later stages of the theory.

\paragraph{Informal description.}
A phase datum consists of an additive object, a family of bounded-degree phase
functions or phase-like behaviours on it, and a specified interaction law that
governs how these behaviours combine. Concrete analytic realisations may be
introduced later, for example modulation and translation systems, Pauli or Weyl
operators, radical polynomial expressions over finite rings, or character-based
constructions, but none of these belong to the minimal input.

\subsubsection*{Analytic admissibility}

The analytic phenomena from which phase data may be extracted are intentionally
broad. Any analytic system that exhibits finite-degree phase behaviour,
functorially detectable interaction patterns, and finite termination is
admissible as input. Different analytic sources may give rise to the same
abstract phase behaviour, and Algebraic Phase Theory is designed to record only
this behavioural core.

Examples of admissible analytic sources include modulation and translation
systems in which failure of rigidity is detected through commutators,
polynomial phases over finite or nilpotent rings detected by additive
derivatives, Weyl or Pauli operator systems with controlled commutation
relations, character twists, and cocycle-type behaviour in additive
combinatorics or ergodic theory. These examples fall naturally into two
regimes. In the weakly admissible regime, phase behaviour is finitely
detectable and the interaction law is functorially meaningful, although one may
lack a faithful analytic operator model. In the strongly admissible regime, the
analytic system admits a concrete realisation, for example by operators on a
Hilbert space, whose commutation behaviour detects defect, depth, and
filtration in a functorial and faithful manner.

\begin{remark}
The collection $\Phi$ belongs to the input data for phase extraction and not to
the algebraic phase itself. Introducing $\Phi$ too early would force premature
commitments regarding the notion of phase, the correct algebraic interaction
law, or the way degree and defect should be detected. At this initial stage we
do not assume whether phases should be viewed as functions, operators, or
equivalence classes, nor do we assume any particular analytic mechanism for
their interaction. The purpose of keeping the analytic input flexible is to
separate raw phase behaviour from the algebraic structure that will later be
extracted from it.
\end{remark}

The aim of this section is therefore not to impose analytic, topological, or
representational structure, but rather to isolate the minimal information
required to speak meaningfully about phase interaction and its algebraic shadow.

\subsubsection*{Additive derivatives and degree}

Let $A$ be an abelian group (or an $R$-module) and let $\phi:A\to R$ be a
function.  For $h\in A$ define the \emph{additive difference operator}
\[
(\Delta_h\phi)(x):=\phi(x+h)-\phi(x).
\]
For $h_1,\dots,h_k\in A$ define the $k$-fold additive derivative inductively by
\[
\Delta_{h_1,\dots,h_k}\phi := \Delta_{h_k}\cdots \Delta_{h_1}\phi .
\]
We say that $\phi$ has \emph{degree at most $d$} if
\[
\Delta_{h_1,\dots,h_{d+1}}\phi \equiv 0
\qquad\text{for all }h_1,\dots,h_{d+1}\in A.
\]
Equivalently, all $(d+1)$-fold additive derivatives vanish identically.

\begin{remark}
The additive difference operator $\Delta_h$ is canonical, translation-invariant,
and functorial.  Any notion of degree or defect compatible with the axioms of
Algebraic Phase Theory is equivalent to one defined via iterated additive
differences.
\end{remark}

\begin{lemma}
Let $A$ be an abelian group, $R$ an abelian group, and let $\phi:A\to R$ be a
function. For $h_1,\dots,h_k\in A$ one has
\[
\Delta_{h_1,\dots,h_k}\phi(x)
=
\sum_{\epsilon\in\{0,1\}^k}
(-1)^{\,k-|\epsilon|}
\,
\phi\!\left(
x+\sum_{i=1}^k \epsilon_i h_i
\right).
\]
\end{lemma}
\begin{proof}
We argue by induction on $k$. For $k=1$ the identity is exactly the definition:
\[
\Delta_{h_1}\phi(x)=\phi(x+h_1)-\phi(x)
=
(-1)^{1-0}\phi(x)+(-1)^{1-1}\phi(x+h_1).
\]

Assume the formula holds for some $k\ge 1$. Set $H=(h_1,\dots,h_k)$ and write
\[
\Delta_{H,h_{k+1}}\phi(x)
=
\Delta_{h_{k+1}}\bigl(\Delta_H\phi\bigr)(x)
=
\Delta_H\phi(x+h_{k+1})-\Delta_H\phi(x).
\]
By the induction hypothesis,
\[
\Delta_H\phi(x+h_{k+1})
=
\sum_{\epsilon\in\{0,1\}^k}
(-1)^{\,k-|\epsilon|}
\,
\phi\!\left(
x+h_{k+1}+\sum_{i=1}^k \epsilon_i h_i
\right),
\]
and
\[
\Delta_H\phi(x)
=
\sum_{\epsilon\in\{0,1\}^k}
(-1)^{\,k-|\epsilon|}
\,
\phi\!\left(
x+\sum_{i=1}^k \epsilon_i h_i
\right).
\]
Subtracting gives
\[
\Delta_{H,h_{k+1}}\phi(x)
=
\sum_{\epsilon\in\{0,1\}^k}
(-1)^{\,k-|\epsilon|}
\left[
\phi\!\left(
x+h_{k+1}+\sum_{i=1}^k \epsilon_i h_i
\right)
-
\phi\!\left(
x+\sum_{i=1}^k \epsilon_i h_i
\right)
\right].
\]
Now identify $\{0,1\}^{k+1}$ with pairs $(\epsilon,\epsilon_{k+1})$ where
$\epsilon\in\{0,1\}^k$ and $\epsilon_{k+1}\in\{0,1\}$. The bracketed difference
is exactly the contribution of the two choices $\epsilon_{k+1}=1$ and
$\epsilon_{k+1}=0$, with signs
\[
(-1)^{\,k-|\epsilon|}
=
(-1)^{\,(k+1)-(|\epsilon|+1)}
\quad\text{and}\quad
-(-1)^{\,k-|\epsilon|}
=
(-1)^{\,(k+1)-|\epsilon|}.
\]
Thus the subtraction recombines into the single sum
\[
\Delta_{h_1,\dots,h_{k+1}}\phi(x)
=
\sum_{\eta\in\{0,1\}^{k+1}}
(-1)^{\,(k+1)-|\eta|}
\,
\phi\!\left(
x+\sum_{i=1}^{k+1} \eta_i h_i
\right),
\]
which is the desired formula for $k+1$.
\end{proof}

\subsubsection*{Defect tensors and defect degree}

The \emph{defect} of a function $\phi:A\to R$ is the obstruction to additivity
measured by its iterated additive derivatives.  Concretely, defect is detected
by the first nonvanishing additive derivative of $\phi$. For a function $\phi:A\to R$ of degree $\le d$, define its \emph{defect degree} by
\[
\deg_{\Def}(\phi)
:=
\min\Bigl\{k\ge 1:\exists\,h_1,\dots,h_k\in A\ \text{with}\
\Delta_{h_1,\dots,h_k}\phi\not\equiv 0\Bigr\},
\]
with the convention $\deg_{\Def}(\phi)=0$ if $\phi$ is additive. The associated \emph{defect tensor} is the family
\[
\Def(\phi)
:=
\Bigl\{\Delta_{h_1,\dots,h_k}\phi(0)\ :\ k=\deg_{\Def}(\phi)\Bigr\}.
\]

\begin{remark}
In finite ring realizations, defect values often lie in the coefficient ring $R$
and are stratified by its radical filtration.  In such cases, the nilpotence
length of $\rad(R)$ bounds the number of distinct defect layers that can be
detected.  This phenomenon reflects the realization, not the abstract phase
structure.
\end{remark}

\subsubsection*{Functorial pullback of phases}

Let $\mathsf{C}$ be a chosen base category of additive objects.  For a morphism
$f:A\to A'$ and a phase $\phi':A'\to R$, define the pullback
\[
f^\ast(\phi') := \phi'\circ f .
\]
This defines a contravariant functorial action of $\mathsf{C}$ on phase data.

\subsubsection*{Weak and strong admissibility}

\begin{definition}\label{def:weakly-admissible-phase-datum}
A \emph{weakly admissible phase datum} consists of a triple $(A,\Phi,\circ)$ such
that:
\begin{enumerate}[label=(W\arabic*)]
\item $A$ is an object of $\mathsf{C}$ and $\Phi(A)$ is a functorial family of
phases $\phi:A\to R$;
\item all phases have uniformly bounded additive degree;
\item the interaction law $\circ$ is specified abstractly and admits functorial
detection of defect and failure of closure at finite depth.
\end{enumerate}
No operator realization, character theory, or analytic structure is assumed.
\end{definition}

\begin{definition}\label{def:admissible-phase-datum}
A \emph{strongly admissible phase datum} is a weakly admissible phase datum
together with a faithful operator realization of phases that detects defect,
filtration, and rigidity functorially.
\end{definition}

\begin{remark}
Strong admissibility is a rigidity enhancement of weak admissibility.
Every strongly admissible phase datum is, by definition, weakly admissible, but
the converse need not hold.

Weak admissibility captures the intrinsic algebraic content of phase interaction
and defect detection, ensuring that defect and filtration are functorially
detectable at finite depth.  However, it does not force uniqueness of defect
propagation or closure: additional interaction data may exist whose behaviour is
not uniquely determined by the axioms.

Conceptually, a weakly admissible phase datum should be viewed as consisting of
two layers:
\begin{enumerate}
\item a canonically determined defect--controlled regime, in which propagation
and interaction are rigid and functorially forced; and
\item additional extension data beyond this regime, whose propagation may fail
or become choice dependent at greater depth.
\end{enumerate}

In strongly admissible regimes, boundary depth is therefore controlled entirely
by defect rank.  By contrast, weak admissibility allows \emph{boundary slippage}:
canonical propagation may fail strictly beyond the defect--generated layers.
Such extension phenomena are the intrinsic source of boundary effects and are
analyzed systematically later in the series, where the canonical strongly
admissible core and its universal property are established at the level of
algebraic Phases.
\end{remark}

\subsubsection*{Operator realizations via generating characters}

One important class of strongly admissible phase data arises from operator
realizations over finite rings.  Here phases are represented as character
weighted multiplication operators on the function space
$\mathcal{H}(A)=\Fun(A,\mathbb{C})$, where a generating additive character
$\chi:R\to\mathbb{C}^\times$ serves as the analytic interface between the
coefficient ring and the operator model.  This construction yields a faithful
realization in which additive differences, defect values and phase interactions
are detected functorially and with maximal resolution.  It therefore provides a
canonical analytic embodiment of the strongly admissible regime. Let $R$ be a finite ring and fix a generating additive character
\[
\chi:R\to\mathbb{C}^\times .
\]
Write $\mathcal{H}(A):=\Fun(A,\mathbb{C})$.  To each phase $\phi:A\to R$ associate
the phase multiplication operator
\[
(M_\phi f)(x):=\chi(\phi(x))\,f(x).
\]

\begin{remark}\label{rem:frobenius-necessity}
For finite rings, the existence of a generating additive character is equivalent
to the ring being Frobenius.  Consequently, whenever phases are realized via
characters as above, the base ring is automatically Frobenius.

This condition is not imposed by Algebraic Phase Theory itself.  Rather, it is a
sufficient hypothesis guaranteeing a faithful operator realization with maximal
defect detectability.  Weak admissibility does not require Frobenius duality, and
phases with delayed or higher-order boundary phenomena need not admit such a
realization.
\end{remark}

\medskip
The abstract notions introduced here are intentionally sparse.  In subsequent
sections we examine both strongly admissible examples, where defect detection is
maximal, and weakly admissible phases in which boundary phenomena appear beyond
the depth detected by defect alone.

\section{Radical Phase Geometry}

Before introducing the radical quadratic model, we record an immediate
structural consequence of the admissible input data.
Over finite rings, the requirement that phase functions be realized via a
\emph{generating} additive character already forces strong algebraic
restrictions on the base ring.

We now turn to the minimal nontrivial example of an algebraic Phase.
This example is not illustrative but foundational: all axioms of Algebraic Phase
Theory are forced by the phenomena encountered here.
It exhibits intrinsic defect, a canonical finite filtration, and a sharp
quadratic structural boundary.
All subsequent abstraction in this paper is driven by the behaviour observed in
this model.

\begin{remark}
In this section we move between analytic and algebraic descriptions.  
Analytically, a phase is a function $\phi:A\to R$ whose additive derivatives
record degree and defect.  Algebraically, the same phase is realized as an
operator $M_\phi$ acting on $\mathcal H(A)$, and defect controls how such
operators interact under composition.  The radical quadratic model is the first
setting in which these two viewpoints coincide.  Analytic defect computations
force the algebraic filtration, and the algebraic interaction law reflects the
analytic degree.  This dual description is an essential feature of the model.
\end{remark}

Let
\[
R := \mathbb{F}_2[u]/(u^2)
\]
be the finite Frobenius ring equipped with its canonical generating additive
character
\[
\chi(a+ub) := (-1)^{b} .
\]
The Jacobson radical $\rad(R)=(u)$ is nontrivial and satisfies $\rad(R)^2=0$.
For each $n\ge 1$, set
\[
\mathcal H(R^n) := \Fun(R^n,\mathbb{C}),
\]
the space of complex-valued functions on the finite $R$-module $R^n$.
A function $\phi:R^n\to R$ is called a \emph{quadratic phase} if its polarization
\[
\mathsf{B}_\phi(x,y)
:= \phi(x+y)-\phi(x)-\phi(y)
\]
is biadditive.
Note that this notion is invariant under adding constants to $\phi$; in
particular, quadratic phases are quadratic functions up to additive constants.
To each quadratic phase $\phi$ we associate the phase multiplication operator
\[
(M_\phi f)(x) := \chi(\phi(x))\,f(x),
\qquad f\in\mathcal H(R^n).
\]

In addition, we consider the translation operators
\[
(T_a f)(x) := f(x+a),
\qquad a\in R^n .
\]

\subsection*{Extraction of admissible phase data}

We now make explicit how the preceding construction fits into the general
framework of admissible phase data.  This identification isolates exactly
which features of the example are structural and which are incidental.

\medskip

\noindent
\textbf{Underlying additive object.}
We take
\[
A := R^n,
\]
viewed as a finite $R$-module and hence as an object of the base category
$\mathsf{C}$ of finite additive groups (or finite $R$-modules).

\medskip

\noindent
\textbf{Phases.}
Let $\Phi(A)$ denote the collection of all quadratic functions
\[
\phi : A \to R
\]
whose polarization $\mathsf{B}_\phi$ is biadditive.  Equivalently, $\Phi(A)$
consists of all functions of additive degree at most $2$ in the sense of
iterated additive differences.

\medskip

\noindent
\textbf{Uniform bounded degree.}
Every $\phi\in\Phi(A)$ satisfies
\[
\Delta_{h_1,h_2,h_3}\phi \equiv 0,
\]
so the family $\Phi(A)$ has uniformly bounded degree with $d=2$.

\begin{example}[Explicit additive derivatives for a sharp quadratic phase]
\label{ex:explicit-derivatives}
Let $R=\mathbb{F}_2[u]/(u^2)$ and $A=R^2$.  Define
\[
\phi:A\to R,\qquad \phi(x_1,x_2)=u\,x_1x_2.
\]
Fix $x=(x_1,x_2)\in A$ and increments
\[
h=(h_1,h_2),\qquad k=(k_1,k_2),\qquad \ell=(\ell_1,\ell_2)\in A.
\]

\medskip
\noindent\textbf{First additive derivative.}
By definition,
\[
(\Delta_h\phi)(x)=\phi(x+h)-\phi(x).
\]
Compute each term:
\[
\phi(x+h)=u(x_1+h_1)(x_2+h_2)
=u(x_1x_2+x_1h_2+h_1x_2+h_1h_2),
\]
and
\[
\phi(x)=u x_1x_2.
\]
Subtracting gives
\[
(\Delta_h\phi)(x)
=u(x_1h_2+h_1x_2+h_1h_2).
\]

\medskip
\noindent\textbf{Second additive derivative.}
By definition,
\[
(\Delta_{h,k}\phi)(x)
=(\Delta_h\phi)(x+k)-(\Delta_h\phi)(x).
\]
Using the formula above,
\[
(\Delta_h\phi)(x+k)
=u\bigl((x_1+k_1)h_2+h_1(x_2+k_2)+h_1h_2\bigr).
\]
Subtracting,
\[
(\Delta_{h,k}\phi)(x)=u(k_1h_2+h_1k_2),
\]
which is independent of $x$.

\medskip
\noindent\textbf{Third additive derivative.}
Since $(\Delta_{h,k}\phi)(x)$ is constant,
\[
(\Delta_{h,k,\ell}\phi)(x)=0
\qquad\text{for all }x,h,k,\ell\in A.
\]

\medskip
\noindent\textbf{Conclusion.}
The second derivative is nonzero for suitable increments, while all third
additive derivatives vanish identically.  Therefore $\phi$ has additive degree
exactly $2$.
\end{example}

\medskip

In the radical phase model over $R=\mathbb{F}_2[u]/(u^2)$, defect is inherently
quadratic: the first nontrivial obstruction to additivity appears at second
order, while all higher additive derivatives vanish identically.  Thus every
non additive quadratic phase in $\Phi(A)$ has defect degree exactly $2$.

\medskip

\noindent
\textbf{Operator realization.}
Fixing the generating additive character
\[
\chi : R \to \mathbb{C}^\times,
\]
each phase $\phi\in\Phi(A)$ acts on $\mathcal H(A)=\Fun(A,\mathbb{C})$ by the
phase multiplication operator $M_\phi$.  In addition, translations $T_a$ act
by pullback along addition in $A$.

\medskip

\noindent
\textbf{Interaction law.}
The interaction law $\circ$ is taken to be operator composition in
$\End(\mathcal H(A))$.  The algebraic phase model $\mathcal P$ is generated by
all operators $M_\phi$ and $T_a$ and closed under this operation.

\begin{remark}
The filtration on the algebraic Phase is not chosen or added. It is forced by
the admissible analytic behaviour. Each analytic phase $\phi$ has a defect
degree determined by the first nonvanishing additive derivative, measured by the
operators $\Delta_{h_1,\dots,h_k}$. When such a phase is realised as an operator
$M_\phi$, the same defect degree determines how far the corresponding operator
lies from the rigid component of the algebraic Phase. Closure under composition
cannot create deeper defect than the analytic behaviour permits, since every
commutation step ultimately reduces to additive differences of $\phi$ and these
differences vanish beyond the analytic defect degree. Because the analytic
defect terminates at finite depth, the algebraic filtration terminates at the
same depth. The filtration is therefore an intrinsic invariant determined
uniquely by the analytic phase behaviour.
\end{remark}

\medskip

\noindent
\textbf{Functoriality.}
For any additive homomorphism
\[
f : R^n \to R^m,
\]
pullback defines a map
\[
f^\ast : \Phi(R^m) \to \Phi(R^n), \qquad
f^\ast(\phi) := \phi \circ f.
\]
Additive derivatives commute with pullback, and degree is preserved.  Thus the
assignment $A \mapsto \Phi(A)$ is functorial.

\medskip

\noindent
\textbf{Conclusion.}
The data
\[
(A,\Phi,\circ)
\]
associated to quadratic phases over $R=\mathbb{F}_2[u]/(u^2)$ therefore form an
admissible phase datum in the sense of
Definition~\ref{def:admissible-phase-datum}.

\medskip
The behaviour observed in the radical quadratic model motivates a general
axiomatic framework.  In the next section we abstract these phenomena into a
small set of structural axioms, designed to capture exactly what is essential
and nothing more.

\begin{corollary}
\label{cor:finite-rings-frobenius}
Let $R$ be a finite ring.  If $R$ admits admissible phase data in the sense of
Definition~\ref{def:admissible-phase-datum} whose operator realization uses a
\emph{generating} additive character $\chi:R\to\C^\times$, then $R$ is a
Frobenius ring.
\end{corollary}

\begin{proof}
By Axiom~IV, the operator realization of an admissible Phase uses a fixed
additive character $\chi:R\to\C^\times$. By hypothesis, this character is \emph{generating}, meaning that $\ker(\chi)$
contains no nonzero left or right ideal of $R$.

It is a standard characterization of finite Frobenius rings that a finite ring
admits a generating additive character if and only if it is Frobenius; see for
example \cite{Wood1999Duality,Honold2001QuasiFrobenius}.  Therefore $R$ is a
Frobenius ring.
\end{proof}

\begin{remark}
This corollary shows that Frobenius duality is not an auxiliary assumption of
Algebraic Phase Theory but a structural consequence of the admissible input
data.  In particular, any finite-ring phase model admitting a generating
character necessarily lies within the class of Frobenius rings.
For background on Frobenius rings and character theory, see
\cite{Wood1999Duality,Honold2001QuasiFrobenius,GreferathSchmidt2000Characters}.
\end{remark}

\medskip

\section{Algebraic Phase Theory: Structural Axioms}

For notational convenience, we summarise an algebraic phase model by the pair
\[
(\mathcal P,\circ),
\]
where $\mathcal P$ denotes the underlying algebraic structure encoding phase
interactions and $\circ$ denotes the distinguished algebraic operations.  
Defect and filtration are not auxiliary data.  They are invariants extracted
from the controlled failure of strict algebraic closure under $\circ$.  
All defect information, canonical filtrations, and intrinsic notions of
complexity are therefore \emph{not imposed as additional structure}.  
Under the standing hypotheses of Algebraic Phase Theory, they are functorially
and uniquely determined by the interaction laws encoded by $(\mathcal P,\circ)$.
Once $(\mathcal P,\circ)$ is fixed, there are no further choices involved in
defining defect, filtration, or complexity, up to canonical equivalence.

Thus the pair $(\mathcal P,\circ)$ introduces no extra data; it is simply a
compact way to refer to the full algebraic phase structure determined by these
operations.  Throughout this paper all algebraic Phases are assumed to satisfy
the following axioms.  These axioms are not arbitrary: each one formalises a
feature that arises unavoidably in the radical quadratic model, and together
they describe the minimal structural conditions under which phase interactions
admit a coherent algebraic description.

\subsection*{Axiom I: Detectable Action Defects}

The interaction law $\circ$ canonically determines a defect degree for each
element of $\mathcal P$, functorially and uniquely. Each algebraic Phase therefore carries a canonical finite filtration
\[
\mathcal{P}_0 \subset \mathcal{P}_1 \subset \cdots \subset \mathcal{P}_d = \mathcal{P},
\]
where $\mathcal{P}_k$ consists of all elements of defect degree at most $k$.
This filtration is intrinsic to $\mathcal{P}$ and is determined entirely by the
defect. Here $d$ denotes the maximal defect degree, so $k$ ranges from $0$ to $d$.

\subsection*{Axiom II: Canonical Algebraic Realization}

The analytic behaviour determines a unique algebraic shadow: any two algebraic
Phases extracted from the same admissible data are canonically equivalent.

\subsection*{Axiom III: Defect-Induced Complexity}

The canonical filtration
\[
\mathcal{P}_0 \subset \mathcal{P}_1 \subset \cdots \subset \mathcal{P}_d
\]
is ordered by defect degree: an element lies in $\mathcal{P}_k$ if and only if
its defect degree is at most $k$.  Thus higher layers correspond exactly to
higher defect.

\subsection*{Axiom IV: Functorial Defect Structure}

Defect is functorial: it is preserved under pullback, under all phase morphisms,
and when passing between analytic behaviour and its algebraic shadow.

\subsection*{Axiom V: Finite Termination}

Defect degree is always finite.\\

\bigskip

We emphasize that Algebraic Phase Theory does not prescribe a unique concrete
realization of $(\mathcal P,\circ)$.  In practice, the same analytic phase
behaviour may admit several natural algebraic realizations, differing for
example in presentation, generating sets, or ambient categories. The axioms are designed to be insensitive to such choices.  Once a realization
$(\mathcal P,\circ)$ satisfies Axioms~I-V, all notions of defect, filtration,
layering, and finite termination are canonically determined and functorial.
Different realizations of the same underlying phase behaviour therefore yield
equivalent algebraic Phases in the sense relevant to this theory.

In particular, Algebraic Phase Theory is a theory of invariants and structural
boundaries, not of preferred models. Having stated the axioms abstractly, we now return to the radical quadratic
model and verify that it satisfies each axiom intrinsically. These axioms apply equally to Phases extracted from weakly or strongly
admissible phase data; admissibility governs the existence and detectability of
structure, not the internal behaviour once a Phase is formed.

\section{Verification of the Axioms in the Radical Phase Model}

We now verify that the radical quadratic phase model over 
$R=\mathbb{F}_{2}[u]/(u^{2})$ satisfies Axioms~I–V of Algebraic Phase Theory.

Let $\mathcal P$ denote the algebra generated by all translation operators
$T_a$ and all quadratic phase multiplication operators $M_\phi$ on
$\mathcal H(A)=\Fun(A,\C)$, closed under composition.  The operation
$\circ$ is composition.

\subsection*{Axiom I: Detectable Action Defects}

In this model, the failure of strict rigidity under composition is detected by
the polarization of a quadratic phase,
\[
\mathsf{B}_\phi(x,y)=\phi(x+y)-\phi(x)-\phi(y).
\]
This form vanishes for linear phases and takes values in the radical $\rad(R)$
precisely when non-rigidity appears.  Thus the interaction law $\circ$ assigns
to each operator a defect degree measuring the first nonvanishing layer of
polarization, and the resulting filtration is exactly
\[
\mathcal P_0 \subset \mathcal P_1 \subset \mathcal P_2=\mathcal P.
\]
Hence Axiom~I holds.

\subsection*{Axiom II: Canonical Algebraic Realization}

The algebraic Phase $\mathcal P$ is obtained functorially from the admissible
quadratic phase data $(A,\Phi)$ via phase multiplication and translation.
Any other algebraic realization of the same analytic behaviour produces the
same operators up to Phase equivalence.  Thus the analytic behaviour determines
a unique algebraic shadow, verifying Axiom~II.

\subsection*{Axiom III: Defect-Induced Complexity}

In the radical model, the defect degree determined above coincides with the
radical depth at which the polarization of a phase becomes nonzero.  This degree
organises $\mathcal P$ into the filtration
\[
\mathcal P_0 \subset \mathcal P_1 \subset \mathcal P_2,
\]
so complexity is completely ordered by defect degree, as required by
Axiom~III.

\subsection*{Axiom IV: Functorial Defect Structure}

Polarization commutes with pullback along homomorphisms
$A\to A'$, and defect degree is preserved under conjugation by translations
and under all admissible morphisms of phases.  Thus defect behaves functorially,
verifying Axiom~IV.

\subsection*{Axiom V: Finite Termination}

Since $\rad(R)^2=0$, polarization cannot produce nonzero values beyond radical
depth two.  Hence defect degree is always finite and the filtration stabilises at
level $2$.  This verifies Axiom~V.

\subsection*{Uniqueness of Defect up to Equivalence}

One may ask whether different ways of measuring non-rigidity could lead to
genuinely different defect invariants in this model.  For example, one could try
to define defect using commutators, higher additive derivatives, or deviations
from exact character behaviour.  Each such construction captures some aspect of
non-rigid interaction. In the radical quadratic phase model these alternatives do not produce distinct
notions.

\medskip
\noindent
\textbf{Key observation.}
\emph{Every functorial defect invariant compatible with finite termination
coincides, up to canonical equivalence, with polarization modulo the Jacobson
radical.}

Any admissible defect invariant must be functorial under automorphisms of $A$,
vanish on rigid phases, detect non-rigidity in composition, and respect the fact
that $\rad(R)^2 = 0$ imposes finite termination.  These requirements force all
candidates to collapse to the same two step hierarchy:
\[
\text{rigid} \quad < \quad \text{non-rigid modulo }\rad(R)
\quad < \quad \text{purely radical}.
\]

Thus in this model all reasonable defect constructions induce the same canonical
filtration and the same defect degrees.  This is an instance of the general
\emph{defect equivalence principle}: different constructions are regarded as
equivalent when they yield the same induced filtration and defect degrees. Having verified Axioms~I–V, we now shift focus from individual operators to the
intrinsic invariants they generate.  In Algebraic Phase Theory it is these
invariants, namely defect, filtration, and structural boundary, that carry the
essential content.

\section{Phase Extraction and Minimality}

This section has two purposes.  
First, it explains how algebraic Phases arise from admissible analytic phase
data by a canonical extraction procedure.  This shows that the axioms of
Algebraic Phase Theory are not imposed externally but emerge directly from the
analytic behaviour of phases.

Second, it identifies the radical quadratic model as the smallest setting in
which genuinely nontrivial algebraic behaviour appears.  Below this level the
theory collapses to rigidity, while at this level a multi step canonical
filtration first emerges and terminates at finite depth.  These two themes are
developed in the subsections that follow.

\subsection*{Phase extraction}

This section describes how one passes from admissible analytic phase data to
an algebraic Phase.  The extraction procedure is canonical: it takes the given
analytic behaviour, packages it into an algebraic interaction law, and reveals
the invariants such as defect, filtration, and finite termination that are
already present.  Theorem \ref{thm:phase-extraction} formalises this
construction and shows that any admissible analytic input yields a unique
algebraic Phase satisfying Axioms I to V.

\begin{theorem}
\label{thm:phase-extraction} Let $(A,\Phi,\circ)$ be an strongly admissible phase datum in the sense of Definition~\ref{def:admissible-phase-datum}. Then one can construct, functorially in $A$, an algebraic Phase $(\mathcal P,\circ)$ realizing the given structural signature and satisfying Axioms~I-V of Algebraic Phase Theory. 
\end{theorem}

\begin{proof}
Starting from strongly admissible phase data, we construct the associated algebraic Phase
and extract its intrinsic invariants, verifying Axioms~I, III, IV and~V.
Uniqueness up to canonical equivalence then follows from Axiom~II.

By strong admissibility (Definition~\ref{def:admissible-phase-datum}), the datum
$(A,\Phi,\circ)$ comes equipped with a faithful operator realization of phases
on a complex vector space $\mathcal H(A)$, in which the distinguished interaction
law $\circ$ is realized as the corresponding operation(s) on operators
(for example, composition in $\End_{\C}(\mathcal H(A))$), and in which defect,
filtration, and rigidity are detected functorially.

Let $\mathcal P(A)\subseteq \End_{\C}(\mathcal H(A))$ denote the smallest subset
containing the phase operators $\{M_\phi:\phi\in\Phi(A)\}$ (together with the
identity operator) and closed under the distinguished operation(s) encoded by
$\circ$.  By construction, $(\mathcal P(A),\circ)$ is closed under the specified
interaction and constitutes an algebraic Phase realizing the given phase data. This assignment is functorial.
For a morphism $f:A\to A'$ in the base category $\mathsf C$, strong admissibility
provides a functorial pullback action on phases $\Phi(A')\to\Phi(A)$,
$\phi'\mapsto \phi'\circ f$. Define
\[
f^\ast:\mathcal P(A')\to \mathcal P(A)
\]
on generators by
\[
f^\ast(M_{\phi'}) := M_{\phi'\circ f}.
\]
By admissibility axiom~\textup{(W1)}, the pullback $\phi'\circ f$ lies in $\Phi(A)$
whenever $\phi'\in\Phi(A')$, so $f^\ast$ is well defined.
Since $f^\ast$ respects the distinguished operation(s) $\circ$, it extends uniquely
to a morphism of algebraic Phases, and functoriality $(g\circ f)^\ast=f^\ast\circ g^\ast$
follows on generators and hence on all of $\mathcal P$.

By Axiom~II, any other algebraic realization of the same admissible phase data
is canonically equivalent to $(\mathcal P(A),\circ)$.  Thus the construction is
unique up to canonical equivalence within the framework of Algebraic Phase Theory.

For each phase $\phi\in\Phi(A)$, defect is measured analytically by the first
nonvanishing iterated additive derivative. The smallest order at which a derivative
does not vanish defines the \emph{defect degree} $\deg_{\Def}(\phi)$, and the
corresponding derivative values form the \emph{defect tensor} $\Def(\phi)$.
Defect vanishes exactly when $\phi$ is additive, identifying the rigid part of
the structure. Defect is functorial.
For a morphism $f:A\to A'$ and phase $\phi':A'\to R$,
\[
\Delta_h(\phi'\circ f)
=
(\Delta_{f(h)}\phi')\circ f.
\]
Thus defect degree and defect tensor are preserved under pullback. Therefore defect
supplies a canonical, functorial invariant measuring deviation from rigidity, as
required by Axioms~I and~IV.

Defect determines a canonical hierarchy in the algebraic Phase.
For each $k\ge 0$, let $\mathcal P_k(A)$ be the subset generated (under the
operation(s) $\circ$) by all $M_\phi$ with $\deg_{\Def}(\phi)\le k$, together with
any designated rigid generators (such as translations) if present.
This yields an increasing filtration
\[
\mathcal P_0(A)\subseteq\mathcal P_1(A)\subseteq\mathcal P_2(A)\subseteq\cdots\subseteq \mathcal P(A),
\]
called the \emph{canonical defect filtration}. Elements of $\mathcal P_0(A)$ are rigid,
and higher layers represent increasing non-rigidity. Functoriality of defect ensures
that for any morphism $f:A\to A'$,
\[
f^\ast(\mathcal P_k(A'))\subseteq \mathcal P_k(A).
\]
Thus defect induces a canonical stratification on $\mathcal P(A)$, minimal precisely
in the rigid regime, as required by Axiom~III.

By the uniform bounded degree hypothesis~\textup{(W2)}, there exists an integer $d$
such that every phase $\phi\in\Phi(A)$ has additive degree at most $d$.
Equivalently, each generator $M_\phi$ has defect degree at most $d$.
Since $\mathcal P(A)$ is generated under $\circ$ by these operators, every element
of $\mathcal P(A)$ lies in the $d$th filtration level:
\[
\mathcal P_d(A)=\mathcal P(A).
\]
Thus the canonical defect filtration terminates after finitely many steps. Finite
termination is therefore a structural consequence of the bounded degree hypothesis,
verifying Axiom~V.

Finally, the construction depends solely on the admissible phase data and the specified
interaction laws, so any two such realizations induce the same defect and canonical
filtration. In particular, $(\mathcal P(A),\circ)$ satisfies Axioms~I, III, IV and~V.
By Axiom~II it is then the unique algebraic Phase determined by the given admissible
phase data, up to canonical equivalence.
\end{proof}

\begin{theorem}
\label{thm:phase-extraction-weak}
Let $(A,\Phi,\circ)$ be a \emph{weakly admissible} phase datum in the sense of
Definition~\ref{def:weakly-admissible-phase-datum}. Then the additive derivative
calculus canonically determines, functorially in $A$, a defect degree function
$\deg_{\Def}:\Phi(A)\to\{0,1,\dots,d\}$ and defect tensors $\Def(\phi)$ for
$\phi\in\Phi(A)$, and hence a canonical increasing defect filtration
\[
\Phi_{\le 0}(A)\subseteq \Phi_{\le 1}(A)\subseteq \cdots \subseteq \Phi_{\le d}(A)=\Phi(A),
\qquad
\Phi_{\le k}(A):=\{\phi\in\Phi(A):\deg_{\Def}(\phi)\le k\}.
\]
This filtration is functorial under pullback: for any morphism $f:A\to A'$ in
$\mathsf C$ one has $f^\ast(\Phi_{\le k}(A'))\subseteq \Phi_{\le k}(A)$.
Moreover, the filtration terminates at depth $d$.
\end{theorem}

\begin{proof}
We begin by defining defect intrinsically using additive derivatives.
For a phase $\phi\in\Phi(A)$, define its \emph{defect degree} $\deg_{\Def}(\phi)$
to be the least integer $k\ge 1$ for which there exist increments
$h_1,\dots,h_k\in A$ such that
\[
\Delta_{h_1,\dots,h_k}\phi \not\equiv 0,
\]
with the convention that $\deg_{\Def}(\phi)=0$ if $\phi$ is additive.
At the defect degree $k=\deg_{\Def}(\phi)$, define the associated
\emph{defect tensor} $\Def(\phi)$ to be the family of values
\[
\Def(\phi)
:=
\{\Delta_{h_1,\dots,h_k}\phi(0)\ :\ h_1,\dots,h_k\in A\}.
\]

By the bounded degree hypothesis of weak admissibility, there exists an integer
$d$ such that
\[
\Delta_{h_1,\dots,h_{d+1}}\phi \equiv 0
\qquad\text{for all }\phi\in\Phi(A)\text{ and all }h_1,\dots,h_{d+1}\in A.
\]
It follows that $\deg_{\Def}(\phi)\le d$ for all $\phi\in\Phi(A)$, and hence that
the defect filtration terminates after finitely many steps at depth $d$. Functoriality of defect follows from the basic identity
\[
\Delta_h(\phi'\circ f)
=
(\Delta_{f(h)}\phi')\circ f,
\]
which holds for any morphism $f:A\to A'$ and phase $\phi':A'\to R$.
Iterating this identity shows that all higher additive derivatives commute with
pullback, so both defect degree and defect tensors are preserved.
The induced defect filtration is therefore functorial by construction.
\end{proof}

\subsection*{Minimality of the radical Phase}

This section explains why the radical quadratic model is the first
nontrivial instance of Algebraic Phase Theory.  It is the smallest setting in
which defect produces more than one layer, the canonical filtration has genuine
depth, and finite termination is internally enforced.  Any simplification of
the analytic input collapses the structure to the rigid regime, while any
strengthening forces additional layers.  This places the radical model as the
minimal base case for the theory.

\begin{definition}\label{def:minimal-phase}
Fix a class $\mathscr{X}$ of algebraic Phases (for example, all Phases extracted
from admissible data in a fixed base category $\mathsf C$).
An algebraic Phase $(\mathcal P,\circ)\in\mathscr{X}$ is called
\emph{minimal genuinely nontrivial} if:
\begin{enumerate}[label=(M\arabic*)]
\item it satisfies Axioms~I-V of Algebraic Phase Theory;
\item its canonical defect filtration has at least two strict inclusions,
\[
\mathcal P_0\subsetneq\mathcal P_1\subsetneq\mathcal P;
\]

There are rigid phases, then genuinely non rigid but still controlled phases,
and only after that the full range of phase interactions. 

\item its filtration terminates at a depth $N\ge 2$ that is minimal among all
Phases in $\mathscr{X}$ satisfying {\rm(M2)};
\item any structural restriction of the input phase data used to build it
(e.g.\ lowering degree, or removing nilpotent layers in the base ring)
collapses the filtration to depth $0$ or $1$.
\end{enumerate}
\end{definition}

\begin{definition}\label{def:radical-phase}
Let $R$ be a finite ring with nontrivial Jacobson radical $\rad(R)\neq 0$ and
let $A=R^n$.  Let $\Phi(A)$ consist of all quadratic phases $\phi:A\to R$
(i.e.\ $\Delta_{h_1,h_2,h_3}\phi\equiv 0$), realized as operators $M_\phi$ on
$\mathcal H(A)=\Fun(A,\C)$ via a fixed generating additive character
$\chi:R\to\C^\times$. The \emph{radical quadratic Phase} is the algebraic Phase obtained from the
admissible datum $(A,\Phi,\circ)$ (with $\circ=$ composition) by the extraction
of Theorem~\ref{thm:phase-extraction}. 
\end{definition}

\medskip
\noindent\textbf{Guiding principle.}
The radical quadratic Phase is the first setting in which defect induces a
genuinely multi step canonical filtration while finite termination is enforced
internally by nilpotence. Any weakening collapses to the rigid or nearly rigid
regime; any strengthening forces additional layers and higher complexity. 

\begin{proposition}
\label{prop:minimality-collapse}
If either {\rm(i)} the phase family $\Phi$ is restricted to degree $\le 1$ phases,
or {\rm(ii)} one works in an operator model in which defect is detected through
the radical filtration of $R$, and the base ring $R$ is reduced (in particular,
a field), then the canonical defect filtration stabilizes at depth at most $1$.
In particular, no intrinsically stratified defect hierarchy can occur in these
settings.

Consequently, the radical quadratic Phase is the minimal setting in which defect,
canonical filtration, and finite termination interact in a genuinely nontrivial
way.
\end{proposition}

\begin{proof}
Assume we are in a setting where the given phase datum admits a strongly admissible
realization, so that an extracted algebraic Phase $(\mathcal P,\circ)$ exists as in
Theorem~\ref{thm:phase-extraction}.
Recall that the canonical filtration $\{\mathcal P_k\}_{k\ge 0}$ is defined from
the defect degree:
\[
\mathcal P_k:=\{T\in\mathcal P:\deg_{\Def}(T)\le k\}.
\]
For generators $\phi\in\Phi(A)$, $\deg_{\Def}(\phi)=0$ if and only if $\phi$ is 
additive (i.e.\ $\Delta_h\phi\equiv 0$ for all $h$), and otherwise
$\deg_{\Def}(\phi)\ge 1$. We treat the two collapse mechanisms separately.

\medskip 
Case (i): $\Phi$ restricted to degree $\le 1$ phases.
Assume every $\phi\in\Phi(A)$ has additive degree $\le 1$, i.e.
$\Delta_{h_1,h_2}\phi\equiv 0$ for all $h_1,h_2\in A$.
Then for each fixed $h\in A$, the first difference $\Delta_h\phi$ has additive
degree $\le 0$ as a function of $x$, hence is constant:
\[
(\Delta_h\phi)(x)=\phi(x+h)-\phi(x)=c_\phi(h)\qquad \text{for some }c_\phi(h)\in R.
\]
Thus either $\phi$ is additive (so $\deg_{\Def}(\phi)=0$), or else some $h$ has
$c_\phi(h)\neq 0$ (so $\deg_{\Def}(\phi)=1$). In particular, every phase generator
$M_\phi$ has defect degree in $\{0,1\}$ (and any designated rigid generators, such as
translations, have defect degree $0$ in the intended applications). Since $\mathcal P$
is generated under $\circ$ by these generators, every $T\in\mathcal P$ is constructible
from generators of defect degree $\le 1$, hence $\deg_{\Def}(T)\le 1$ for all $T$.
Therefore $\mathcal P_1=\mathcal P$ (and if all phases are additive then already
$\mathcal P_0=\mathcal P$). Hence the filtration stabilizes at depth $0$ or $1$.

\medskip
Case (ii): the base ring $R$ is reduced, with defect measured through radical depth.
Assume $R$ is finite and reduced (no nonzero nilpotent elements). Then $R$ is a finite
product of finite fields, so its Jacobson radical vanishes:
\[
\rad(R)=0.
\]
In the operator models considered in this paper in which defect is detected through the
radical filtration of $R$, the defect data land in
\[
\rad(R)\supseteq \rad(R)^2\supseteq\cdots.
\]
When $\rad(R)=0$ there are no nontrivial radical layers, so this mechanism collapses:
the only possible defect strata are the defect-zero (rigid) stratum and its complement.
Equivalently, the canonical filtration has at most one nontrivial step, so it stabilizes
at depth $0$ or $1$.
\end{proof}

\begin{theorem}\label{thm:radical-minimality}
In the radical quadratic setting $R=\F_2[u]/(u^2)$, the extracted Phase
$(\mathcal P,\circ)$ is minimal genuinely nontrivial in the sense of
Definition~\ref{def:minimal-phase}, with $\mathscr{X}$ taken to be the class of
Phases extracted from admissible data over finite $R$-modules with the fixed
choice of generating character.
\end{theorem}
 
\begin{proof}
Let $R=\F_2[u]/(u^2)$ and $A=R^n$, and let $(\mathcal P,\circ)$ be the extracted
Phase of Definition~\ref{def:radical-phase}. By construction and
Theorem~\ref{thm:phase-extraction}, $(\mathcal P,\circ)$ satisfies Axioms~I-V,
so {\rm(M1)} holds. 

To see {\rm(M2)}, note that rigid (additive) phases exist, so $\mathcal P_0\neq 0$.
Moreover, there exist genuinely quadratic phases with nontrivial polarization
(equivalently, nonvanishing second additive derivative), so $\mathcal P_0\subsetneq\mathcal P$.
Since $\rad(R)^2=0$, no defect beyond quadratic depth can occur, and hence the
canonical filtration terminates at depth $2$, i.e.\ $\mathcal P_2=\mathcal P$.
Thus the defect filtration is nontrivial and terminates at the first possible
depth compatible with a genuinely quadratic defect layer.

For {\rm(M3)} and {\rm(M4)}, Proposition~\ref{prop:minimality-collapse} shows that
below the radical quadratic setting, either by restricting to degree $\le 1$
phases or by passing to a reduced ring, the canonical filtration stabilizes at
depth at most $1$. Consequently, no Phase in $\mathscr X$ satisfying {\rm(M2)}
can terminate at depth $<2$, and any such restriction collapses the hierarchy.
Therefore $(\mathcal P,\circ)$ is minimal genuinely nontrivial.
\end{proof}

\begin{remark}
The minimality statement is structural rather than categorical.
It asserts that the radical quadratic setting is the first point, within the
class $\mathscr X$ of Phases satisfying Axioms~I--V, at which intrinsic defect
strata appear beyond the nearly rigid regime while still admitting finite
termination.
No claim of absolute minimality is made outside this axiomatic framework.
\end{remark}

\section{Canonical Filtration and Finite Termination}\label{sec:canonical-filtration}

This section makes precise the intrinsic hierarchy imposed by defect.
Given an algebraic Phase, defect does not merely distinguish rigid from
non rigid behaviour; it canonically organizes all phase interactions into
levels of increasing complexity.
We show that this organization takes the form of a uniquely determined,
functorial filtration, and that under the axioms of Algebraic Phase Theory
this filtration must terminate after finitely many steps.
In particular, complexity is not an auxiliary grading or external choice,
but a structure forced internally by defect and bounded phase degree. 

\subsection*{Defect degree and the induced filtration}

The purpose of this subsection is to extract a numerical notion of
complexity from defect and to show that it canonically induces a filtration
of the Phase.
Intuitively, the defect degree of an element measures how far it lies from
the rigid regime, in terms of the depth at which its non rigidity becomes
detectable under the distinguished operations.
The key point is that this notion of complexity is intrinsic and functorial:
once defect is fixed, there is a unique way to stratify the Phase into
increasing levels, with the defect-zero elements forming the rigid core.

\begin{definition}\label{def:defect-degree-phase}
Let $(\mathcal P,\circ)$ be an algebraic Phase with defect object $\Def$.
For $T\in\mathcal P$, define $\deg_{\Def}(T)$ to be the least $k$ such that,
with respect to the defect-induced complexity stratification supplied by
Axiom~III, all defects generated by $T$ under the distinguished operations are
detected at levels $\le k$. 
\end{definition}

We refer to the filtration $\{\mathcal P_k\}_{k\ge 0}$ as the \emph{canonical
defect filtration} of $(\mathcal P,\circ)$.

\begin{definition}\label{def:canonical-filtration}
For $k\ge 0$ define
\[
\mathcal P_k:=\{T\in\mathcal P:\deg_{\Def}(T)\le k\}.
\]
\end{definition}

\begin{theorem}\label{thm:canonical-filtration}
Let $(\mathcal P,\circ)$ be an algebraic Phase satisfying Axioms~I-V, with
defect structure $\Def$ and defect degree $\deg_{\Def}$ as in
Definition~\ref{def:defect-degree-phase}.
Let $\{\mathcal P_k\}_{k\ge 0}$ denote the associated defect filtration.
Then:
\begin{enumerate}[label=(\alph*)] 
\item $\{\mathcal P_k\}_{k\ge 0}$ is an increasing filtration,
\[
\mathcal P_0\subseteq \mathcal P_1\subseteq \mathcal P_2\subseteq \cdots \subseteq \mathcal P;
\]
\item the filtration is functorial: if $F:(\mathcal P,\circ)\to(\mathcal Q,\circ)$ is a
Phase morphism, then $F(\mathcal P_k)\subseteq \mathcal Q_k$ for all $k$;
\item the filtration is uniquely determined by the defect structure, in the sense that
any filtration defined as sublevel sets of $\deg_{\Def}$ coincides with
$\{\mathcal P_k\}$;
\item $\mathcal P_0$ coincides with the rigid (defect zero) substructure.
\end{enumerate} 
\end{theorem}

\begin{proof}
(a) By definition, $\mathcal P_k$ consists of those $T$ with $\deg_{\Def}(T)\le k$.
If $k\le \ell$ and $T\in\mathcal P_k$, then $\deg_{\Def}(T)\le k\le \ell$, hence
$T\in\mathcal P_\ell$. Therefore $\mathcal P_k\subseteq \mathcal P_\ell$ for
$k\le \ell$, giving an increasing filtration.

\medskip
(b) Fix $k\ge 0$ and let $T\in\mathcal P_k$. By definition,
$\deg_{\Def}(T)\le k$.
Since $F:(\mathcal P,\circ)\to(\mathcal Q,\circ)$ is a Phase morphism,
it is defect compatible, so defect degree does not increase under $F$:
\[
\deg_{\Def}(F(T))\le \deg_{\Def}(T).
\]
Hence $\deg_{\Def}(F(T))\le k$, which by definition means
$F(T)\in\mathcal Q_k$.
Since $T\in\mathcal P_k$ was arbitrary, this shows
\[
F(\mathcal P_k)\subseteq \mathcal Q_k
\qquad\text{for all }k\ge 0.
\] 

\medskip
(c) The sets $\mathcal P_k$ are \emph{defined} as sublevel sets of the intrinsic
defect degree $\deg_{\Def}$. Therefore, once the defect degree is fixed (as part
of the defect structure extracted from $(\mathcal P,\circ)$), there is no
freedom: any filtration defined by the rule ``level $\le k$'' must coincide with
$\{\mathcal P_k\}$.

More precisely, suppose $\{\mathcal P'_k\}_{k\ge 0}$ is any other filtration with
the same defining property, namely that for all $k\ge 0$ and all $T\in\mathcal P$,
\[
T\in\mathcal P'_k \quad\text{if and only if}\quad \deg_{\Def}(T)\le k.
\]
Then for every $k$ and every $T\in\mathcal P$, we have
\[
T\in\mathcal P_k
\;\Longleftrightarrow\;
\deg_{\Def}(T)\le k
\;\Longleftrightarrow\;
T\in\mathcal P'_k.
\]
Hence $\mathcal P'_k=\mathcal P_k$ for all $k$. Therefore, once the defect degree is fixed, the canonical filtration is uniquely
determined by it.

\medskip
(d) By definition,
\[
\mathcal P_0=\{T\in\mathcal P:\deg_{\Def}(T)=0\}.
\]
Elements of defect degree zero are exactly those for which no defect is ever
detected under the distinguished operations. By Axiom~III, this is precisely
the notion of rigidity: defect induced complexity is minimal exactly on elements
with vanishing defect.

Thus $\mathcal P_0$ consists exactly of the rigid elements of $\mathcal P$.
Moreover, any element with nontrivial defect necessarily has positive defect
degree and therefore lies outside $\mathcal P_0$. Hence $\mathcal P_0$ is the
maximal substructure on which the defect invariant vanishes, and coincides with
the rigid regime singled out in Axioms~I and~III. 
\end{proof}

\subsection*{Finite termination and defect-degree bounds}

Having established that defect induces a canonical filtration, we now
explain why this hierarchy must be finite.  
The guiding principle is that defect can only detect non rigidity up to the
maximum degree present in the underlying phase data.  
Once all generators have bounded defect degree, no genuinely new defect can
be produced by iterating the distinguished operations.  
Finite termination is therefore not an additional axiom or constraint, but
a structural consequence of bounded phase degree.

\begin{proposition}\label{prop:termination-bounded-degree}
Suppose $(\mathcal P,\circ)$ arises by phase extraction from an admissible phase
datum $(A,\Phi,\circ)$ with uniform degree bound $d$ (axiom (W3) in
Definition~\ref{def:admissible-phase-datum}). Then the canonical filtration
terminates at depth $d$, i.e. 
\[
\mathcal P_d=\mathcal P.
\]
\end{proposition}

\begin{proof}
Recall that by definition
\[
\mathcal P_d=\{T\in\mathcal P:\deg_{\Def}(T)\le d\}.
\]
Thus, to show $\mathcal P_d=\mathcal P$, it suffices to prove that every
$T\in\mathcal P$ has defect degree at most $d$. By axiom (E3), every phase $\phi\in\Phi(A)$ has additive degree at most $d$, i.e.
\[
\Delta_{h_1,\dots,h_{d+1}}\phi\equiv 0 \qquad \text{for all } h_i\in A.
\]
If $\phi$ is additive then $\deg_{\Def}(\phi)=0$. Otherwise, by definition of
additive degree, there exists a smallest $k\ge 1$ such that some $k$-fold
derivative does not vanish identically, and necessarily $k\le d$. Hence in all 
cases
\[
\deg_{\Def}(\phi)\le d.
\]

By definition of defect degree on generators in the extracted Phase, this yields
\[
\deg_{\Def}(M_\phi)\le d \qquad \text{for all } \phi\in\Phi(A).
\]
Any additional designated rigid generators (such as translations in the intended
applications) have defect degree $0$ by definition. By construction of the extracted Phase (Theorem~\ref{thm:phase-extraction}),
$\mathcal P$ is generated under the distinguished operations $\circ$ by these
generators. Therefore every element $T\in\mathcal P$ can be expressed using 
generators of defect degree at most $d$.

By Axiom~III (defect induced complexity), the defect degree $\deg_{\Def}(T)$ is 
the least $k$ for which $T$ is constructible from generators of defect degree
$\le k$ using the distinguished operations. Since $k=d$ is always admissible,
it follows that $\deg_{\Def}(T)\le d$ for all $T\in\mathcal P$. Hence $\mathcal P\subseteq \mathcal P_d$. Since $\mathcal P_d$ is defined as a subset of $\mathcal P$, namely
\[
\mathcal P_d=\{T\in\mathcal P:\deg_{\Def}(T)\le d\},
\]
we have $\mathcal P_d\subseteq\mathcal P$ by definition. Combined with the
previous inclusion $\mathcal P\subseteq\mathcal P_d$, this yields
\[
\mathcal P_d=\mathcal P.
\]
\end{proof}

\begin{remark}
For $R=\F_2[u]/(u^2)$ the filtration terminates at depth $2$ for structural
reasons (radical length $2$), independent of $n$.
\end{remark}

\section{Structural Boundary Theorem}\label{sec:boundary}  

This section identifies an intrinsic structural boundary for Algebraic Phase
Theory. We formalize the notion of extending an algebraic Phase while preserving
its defect calculus and canonical filtration, and show that finite termination
forces a sharp limitation: any genuinely new extension must introduce defect
of strictly higher complexity. In particular, the terminating depth of the
canonical filtration is not an artifact of presentation but a genuine boundary
of the theory.

\subsection*{Extensions and boundary depth}

We first develop an abstract framework for extending algebraic Phases in a
defect compatible manner. This allows us to formulate the notion of boundary
depth, which measures the maximal defect complexity that can occur without
forcing new defect strata. The main result of this subsection shows that, once
the canonical filtration terminates, any proper extension that agrees with the
original Phase up to that depth must necessarily introduce elements of higher
defect degree.

\begin{definition}\label{def:phase-extension}
An \emph{extension} of an algebraic Phase $(\mathcal P,\circ)$ is an injective
Phase morphism $i:\mathcal P\hookrightarrow \widetilde{\mathcal P}$ such that
the distinguished operations and defect structure restrict along $i$ and the
induced filtration on $i(\mathcal P)$ agrees with the canonical filtration.
\end{definition} 

\begin{lemma}\label{lem:extension-subphase}
Let $i:\mathcal P\hookrightarrow \widetilde{\mathcal P}$ be an extension in the
sense of Definition~\ref{def:phase-extension}. Then:
\begin{enumerate}[label=\textup{(\arabic*)}]
\item The image $i(\mathcal P)\subseteq \widetilde{\mathcal P}$ is a subphase: it
is closed under the distinguished operations $\circ$, and the defect structure
on $\widetilde{\mathcal P}$ restricts to the defect structure of $\mathcal P$.
\item The map $i:\mathcal P\to i(\mathcal P)$ is an isomorphism of algebraic
Phases, with respect to the induced operations and defect.
\item For each $k\ge 0$, the filtration on $i(\mathcal P)$ agrees with that of
$\mathcal P$, in the sense that 
\[
i(\mathcal P_k)=i(\mathcal P)\cap \widetilde{\mathcal P}_k,
\]
where $\{\widetilde{\mathcal P}_k\}$ denotes the canonical filtration on
$\widetilde{\mathcal P}$.
\end{enumerate}
In particular, $\mathcal P$ embeds into $\widetilde{\mathcal P}$ as a full
subphase: no new relations, defects, or filtration levels are introduced on
elements coming from $\mathcal P$. 
\end{lemma}

\begin{proof}
(1). Since $i$ is a Phase morphism, it preserves the distinguished operations $\circ$.
Thus for any $x,y\in\mathcal P$ we have
\[
i(x\circ y)=i(x)\circ i(y),
\]
so $i(\mathcal P)$ is closed under $\circ$. By definition of extension, the
defect structure on $\widetilde{\mathcal P}$ restricts along $i$ to the defect
structure of $\mathcal P$.

\medskip
(2). Because $i$ is injective, the map $i:\mathcal P\to i(\mathcal P)$ is bijective.
Since it preserves the distinguished operations and defect data, it is an
isomorphism of algebraic Phases.

(3). By the extension hypothesis, the induced filtration on $i(\mathcal P)$ agrees
with the canonical filtration of $\mathcal P$. Equivalently, an element
$i(T)\in i(\mathcal P)$ lies in filtration level $\le k$ inside
$\widetilde{\mathcal P}$ if and only if $T\in\mathcal P_k$. This is exactly the
identity
\[
i(\mathcal P_k)=i(\mathcal P)\cap \widetilde{\mathcal P}_k.
\]
\end{proof}

\begin{definition}\label{def:boundary-depth}
Let $(\mathcal P,\circ)$ be an algebraic Phase whose canonical defect filtration
terminates at depth $N$, i.e.\ $\mathcal P_N=\mathcal P$. We say that $\mathcal P$
has \emph{boundary depth} $N$ if for every extension
$i:\mathcal P\hookrightarrow \widetilde{\mathcal P}$ preserving the defect
structure and Axioms~I-V, either $\widetilde{\mathcal P}=i(\mathcal P)$ or
$\widetilde{\mathcal P}$ contains an element of defect degree strictly greater
than $N$.
\end{definition}

\begin{definition}\label{def:tight-extension}
Let $(\mathcal P,\circ)$ be a Phase with canonical filtration $\{\mathcal P_k\}$, 
and let $i:\mathcal P\hookrightarrow \widetilde{\mathcal P}$ be an extension in
the sense of Definition~\ref{def:phase-extension}. Fix $N\ge 0$.
We say that $i$ is \emph{tight up to depth $N$} if every element of
$\widetilde{\mathcal P}$ whose defect degree is at most $N$ lies in the
$\circ$-subalgebra generated by $i(\mathcal P)$. Equivalently, 
\[
\widetilde{\mathcal P}_k
=
\bigl\langle i(\mathcal P_k)\bigr\rangle_{\circ}
\qquad\text{for all }k\le N.
\]
In particular, any element of $\widetilde{\mathcal P}$ not generated by
$i(\mathcal P)$ must have defect degree strictly greater than $N$.
\end{definition}

\begin{theorem}\label{thm:finite-boundary}
Let $(\mathcal P,\circ)$ be an algebraic Phase satisfying Axioms~I-V, and assume
its canonical defect filtration terminates at depth $N$, i.e.\ $\mathcal P_N=\mathcal P$.
Let $i:\mathcal P\hookrightarrow \widetilde{\mathcal P}$ be an extension in the
sense of Definition~\ref{def:phase-extension} which is tight up to depth $N$ in
the sense of Definition~\ref{def:tight-extension}.  If $i$ is proper, then
$\widetilde{\mathcal P}$ contains an element of defect degree $>N$ (equivalently,
an element whose defect is not detected within the terminating filtration of 
$\mathcal P$).
\end{theorem}

\begin{proof}
The argument is a contrapositive style contradiction: we show that if an
extension adds no new defect depth beyond $N$, then it cannot add any genuinely
new elements. Therefore, any \emph{proper} extension must introduce an element
whose defect degree is $>N$. 

Let $i:\mathcal P\hookrightarrow \widetilde{\mathcal P}$ be a tight extension up
to depth $N$, and suppose for contradiction that $\widetilde{\mathcal P}$ has no
new defect depth; that is, every element of $\widetilde{\mathcal P}$ has defect
degree $\le N$. By definition of the canonical filtration this means
\[
\widetilde{\mathcal P}=\widetilde{\mathcal P}_N.
\]
By tightness at level $N$ we have
\[
\widetilde{\mathcal P}_N=\bigl\langle i(\mathcal P_N)\bigr\rangle_{\circ}.
\]
Since $\mathcal P_N=\mathcal P$, this becomes
\[
\widetilde{\mathcal P}
=\widetilde{\mathcal P}_N
=\bigl\langle i(\mathcal P)\bigr\rangle_{\circ}.
\]
But $i(\mathcal P)$ is already closed under the distinguished operations $\circ$
(because $i$ is a Phase morphism and $\mathcal P$ is closed under $\circ$), so
generating under $\circ$ adds nothing:
\[
\bigl\langle i(\mathcal P)\bigr\rangle_{\circ}=i(\mathcal P).
\]
Hence $\widetilde{\mathcal P}=i(\mathcal P)$, so the extension does not add any
new elements. This contradicts the assumption that the extension is proper.
Therefore our assumption was false, and there exists
$S\in\widetilde{\mathcal P}$ with $\deg_{\Def}^{\widetilde{\mathcal P}}(S)>N$.
\end{proof}

\subsection*{Quadratic boundary for the radical Phase}

We now apply the general boundary formalism to the radical quadratic Phase.
Exploiting the fact that quadratic phases are completely controlled by second
additive derivatives, we show that any extension incorporating genuinely cubic
interactions necessarily produces new, irreducible defect invariants. As a
consequence, the radical quadratic Phase admits a sharp structural boundary at
quadratic depth: beyond this level, finite termination cannot be preserved.

\begin{lemma}\label{lem:quadratic-second-derivatives}
Let $A$ be an additive object and let $\phi:A\to R$ have additive degree $\le 2$.
Then any defect datum extracted functorially from the additive derivatives of
$\phi$ is determined by the quadratic polarization 
\[
B_\phi(h_1,h_2):=\Delta_{h_1,h_2}\phi(0).
\]
Modulo additive phases, no further defect information exists.
\end{lemma}

\begin{proof}
Fix $h_1,h_2\in A$. Since $\deg(\phi)\le 2$, the third additive derivative
vanishes identically, so for every $x\in A$ we have
\[
0=\Delta_{x,h_1,h_2}\phi(0)
=\Delta_x\bigl(\Delta_{h_1,h_2}\phi\bigr)(0)
=\Delta_{h_1,h_2}\phi(x)-\Delta_{h_1,h_2}\phi(0).
\]
Thus $\Delta_{h_1,h_2}\phi(x)$ is independent of $x$ and equals
$\Delta_{h_1,h_2}\phi(0)=B_\phi(h_1,h_2)$. 

Now let $\mathcal I(\phi)$ be any datum obtained functorially from the family of
iterated additive derivatives of $\phi$ (for example, any construction that only
uses the collection of values $\Delta_{h_1,\dots,h_k}\phi(0)$ for various $k$).
Because all derivatives of order $\ge 3$ vanish, $\mathcal I(\phi)$ can depend
only on the first and second derivatives. The second derivatives are exactly the
polarization $B_\phi$, and the first derivatives record the additive
(affine linear) part. Therefore $\mathcal I(\phi)$ is determined by
$\{\Delta_h\phi(0)\}_{h\in A}$ and $B_\phi$. 

In particular, since $\deg(\phi)\le 2$, the polarization $B_\phi$ captures all
non additive defect information associated to $\phi$ and is the unique quadratic
invariant arising from its additive derivatives. 

Finally, modulo additive phases (i.e.\ after identifying phases differing by an
additive function), the first derivative data becomes invisible, and the
remaining defect information factors through $B_\phi$ alone. 
\end{proof}
Thus quadratic phases are maximally rigid from the defect point of view: once
additive effects are factored out, quadratic polarization exhausts all
functorial defect data.

\begin{lemma}
\label{lem:cubic-new-invariants}
Let $A$ be an additive object and let $\psi:A\to R$ have additive degree $\le 3$.
If $\psi$ is genuinely cubic (i.e.\ $\Delta_{h_1,h_2,h_3}\psi$ is not identically
zero), then the third additive derivative at the origin 
\[
T_\psi(h_1,h_2,h_3):=\Delta_{h_1,h_2,h_3}\psi(0)
\]
defines a functorial invariant which is not determined by the first and
second derivative data (in particular, not by quadratic polarization).
\end{lemma}

\begin{proof}
We show that the cubic tensor $T_\psi$ contains genuinely new information not
visible at quadratic depth by exhibiting two phases that agree at all defect
levels $\le 2$ but differ in their cubic defect data. For any additive homomorphism $f:A\to A'$ and any $\psi':A'\to R$, a direct
unwinding of iterated differences gives
\[
\Delta_{h_1,h_2,h_3}(\psi'\circ f)(0)
=
\Delta_{f(h_1),f(h_2),f(h_3)}\psi'(0).
\]
Hence
\[
T_{\psi'\circ f}
=
T_{\psi'}\circ(f\times f\times f),
\]
so $T_\psi$ is functorial under pullback.

It suffices to construct two phases of additive degree $\le 3$ whose defect data
agree up to depth $2$ but whose cubic tensors differ. Let
\[
\psi_0 \equiv 0
\]
denote the zero phase. Then $\psi_0$ has additive degree $0\le 3$, and its cubic
tensor is identically zero:
\[
T_{\psi_0}(h_1,h_2,h_3)
=
\Delta_{h_1,h_2,h_3}\psi_0(0)
\equiv 0.
\]

Now work in the setting $A=R^3$ (or more generally $A=R^n$ with $n\ge 3$) and
define
\[
\psi(x_1,x_2,x_3):=x_1x_2x_3.
\]
This function has additive degree $\le 3$ and is genuinely cubic. We first observe that $\psi$ and $\psi_0$ are indistinguishable at quadratic
depth.  Indeed, for any $a,b\in A$, the function
$x\mapsto\Delta_{a,b}\psi(x)$ has additive degree $\le 1$, hence is affine linear
in $x$. Since $\psi$ is homogeneous cubic, every term of
$\Delta_{a,b}\psi(x)$ contains at least one factor of $x$, and therefore
\[
\Delta_{a,b}\psi(0)=0.
\]
Thus
\[
\Delta_{a,b}\psi(0)=\Delta_{a,b}\psi_0(0)
\qquad\text{for all }a,b\in A,
\]
so $\psi$ and $\psi_0$ have identical first and second derivative data at the 
origin.

However, their cubic tensors differ. Choosing standard basis vectors
$e_1,e_2,e_3\in A$, a direct computation yields
\[
T_\psi(e_1,e_2,e_3)
=
\Delta_{e_1,e_2,e_3}\psi(0)=1,
\]
whereas $T_{\psi_0}\equiv 0$. Hence
\[
T_\psi \neq T_{\psi_0}.
\]

Therefore, although $\psi$ and $\psi_0$ agree at all defect levels $\le 2$, they
are distinguished at defect depth $3$. This shows that the cubic tensor $T_\psi$
cannot be determined by quadratic polarization data (nor by adjoining affine
data), completing the proof.
\end{proof}

\begin{theorem}\label{thm:quadratic-boundary}
For $R=\F_2[u]/(u^2)$, any extension of the radical quadratic Phase that
incorporates genuinely cubic interactions necessarily introduces defect strata
beyond the terminating depth of the radical filtration. In particular, quadratic depth is the final terminating level; any genuine
extension beyond it necessarily introduces new defect layers.
\end{theorem}

\begin{proof}
Work in the radical quadratic Phase $(\mathcal P,\circ)$ over $R=\F_2[u]/(u^2)$.
Since $\rad(R)=(u)$ and $\rad(R)^2=0$, the defect-induced filtration of $\mathcal P$
terminates at depth $2$:
\[
\mathcal P_2=\mathcal P. 
\]

Let $i:\mathcal P\hookrightarrow \widetilde{\mathcal P}$ be an extension in the sense
of Definition~\ref{def:phase-extension}. Assume that $\widetilde{\mathcal P}$ genuinely
incorporates cubic interactions, i.e.\ it contains an element $M_\psi$ arising from some
$\psi:A\to R$ of additive degree $\le 3$ with $\deg(\psi)=3$. By Lemma~\ref{lem:cubic-new-invariants}, the third-derivative tensor
\[
T_\psi(h_1,h_2,h_3):=\Delta_{h_1,h_2,h_3}\psi(0)
\]
is a functorial invariant which is not determined by quadratic polarization data.
On the other hand, by Lemma~\ref{lem:quadratic-second-derivatives}, any defect datum 
extracted from degree $\le 2$ phases factors through second additive derivatives, i.e.\
through quadratic polarization.

Suppose for the sake of contradiction that $M_\psi$ were detected within quadratic depth in the
canonical filtration of $\widetilde{\mathcal P}$, namely that
\[
\deg_{\Def}^{\widetilde{\mathcal P}}(M_\psi)\le 2
\qquad\text{(equivalently, } M_\psi\in \widetilde{\mathcal P}_2\text{).}
\]
Then every defect invariant generated from $M_\psi$ under the distinguished operations
would be detected at levels $\le 2$. In particular, the functorial invariant $T_\psi$
would have to be determined by the same quadratic polarization calculus governing depth
$\le 2$ defect. This contradicts Lemma~\ref{lem:cubic-new-invariants}. Therefore
\[
\deg_{\Def}^{\widetilde{\mathcal P}}(M_\psi)>2,
\]
so $\widetilde{\mathcal P}$ contains defect strata beyond the terminating depth of
$\mathcal P$. Hence the structural boundary in the radical quadratic setting occurs 
at quadratic depth.
\end{proof}

This does not preclude extensions of higher defect depth; rather, it asserts that
any such extension necessarily introduces new defect strata beyond quadratic level.

\begin{corollary}[No coherent higher degree extension]\label{cor:no-higher-degree}
Any functorial extraction of bounded-degree phase interactions over 
$R=\F_2[u]/(u^2)$ that preserves finite termination must factor through quadratic
polarization.
\end{corollary}

\begin{proof}
Let a functorial extraction procedure over $R=\F_2[u]/(u^2)$ produce an algebraic
Phase $(\widetilde{\mathcal P},\circ)$ satisfying Axioms~I-V, and assume that the
resulting Phase has finite termination, i.e.\ its canonical defect filtration 
terminates at some finite depth.

In the radical quadratic setting, the canonical filtration of the extracted
Phase $\mathcal P$ already terminates at depth $2$. Suppose that the extraction
procedure genuinely incorporates cubic interactions, meaning that
$\widetilde{\mathcal P}$ contains defect data arising from phases of additive
degree $3$ which are not determined by quadratic polarization.

By Theorem~\ref{thm:quadratic-boundary}, any extension of the radical quadratic
Phase that incorporates genuinely cubic interactions must introduce defect
strata beyond quadratic depth. Consequently, the canonical filtration of 
$\widetilde{\mathcal P}$ cannot terminate at depth $2$.

This contradicts the assumption that the extraction preserves finite
termination at the quadratic level. Therefore, no functorial extraction over
$R=\F_2[u]/(u^2)$ which preserves finite termination can genuinely depend on
cubic data.

Equivalently, any such extraction must factor through the quadratic defect
calculus, i.e.\ through quadratic polarization (second additive derivatives).
\end{proof}

\section{Equivalence, Presentations, and Robustness}\label{sec:equivalence}

This section addresses robustness in Algebraic Phase Theory.
The axioms and constructions developed earlier are intended to capture
intrinsic algebraic structure forced by phase interactions, rather than
artefacts of particular realizations or modeling choices.
Accordingly, we clarify when two algebraic Phases should be regarded as
\emph{the same}, and when distinct ways of measuring defect encode the
same structural information.
The purpose of this section is to formalize these equivalence notions and
to show that the theory is insensitive to presentation level choices.

\subsection*{Equivalence of Phases}

We first describe when two algebraic Phases are considered equivalent. 
Intuitively, two Phases should be identified if they encode the same
hierarchy of rigidity and non-rigidity, even if they arise from different
operator realizations, generating sets, or ambient categories.
The notion of equivalence adopted here requires that the two Phases admit
defect-compatible, filtration-preserving morphisms inducing isomorphisms 
on their associated graded objects.
This ensures that equivalence is determined entirely by intrinsic defect
structure and complexity, not by presentation.

\begin{definition}\label{def:phase-equivalence}
Two algebraic Phases are \emph{equivalent} if there exist defect-compatible,
filtration-preserving morphisms between them inducing isomorphisms on associated
graded objects.  In particular, they agree on defect strata up to canonical 
identification.
\end{definition}

\begin{proposition}\label{prop:presentation-insensitive}
The algebraic Phase object extracted from a given admissible phase datum is
independent, up to equivalence, of the choice of operator realization or
presentation.
\end{proposition} 

\begin{proof}
Let $(\mathcal P^{(1)},\circ)$ and $(\mathcal P^{(2)},\circ)$ be two algebraic
Phases obtained from the same admissible phase datum but via different operator
realizations (e.g.\ different choices of generating sets, ambient operator
algebras, or equivalent generating additive characters).

By admissibility and functoriality of phase extraction
(Theorem~\ref{thm:phase-extraction}), both constructions are functorial in the
underlying phase datum and use the same defect calculus: additive derivatives,
defect degree, and defect detection are defined intrinsically on phase functions 
before realization as operators. Consequently, there exist canonical Phase morphisms
\[
F_{12}:\mathcal P^{(1)}\to\mathcal P^{(2)}, \qquad
F_{21}:\mathcal P^{(2)}\to\mathcal P^{(1)}
\]
induced by identifying generators corresponding to the same phase functions.
These morphisms preserve the distinguished operations and are defect-compatible
by construction (Axiom~IV).

Since defect degree is defined intrinsically from defect data and the
distinguished operations (Axiom~III), both $F_{12}$ and $F_{21}$ preserve defect 
degree:
\[
\deg_{\Def}(F_{12}(T))=\deg_{\Def}(T), \qquad
\deg_{\Def}(F_{21}(S))=\deg_{\Def}(S).
\]
Therefore they preserve the canonical filtration levelwise:
\[
F_{12}(\mathcal P^{(1)}_k)\subseteq \mathcal P^{(2)}_k,
\qquad
F_{21}(\mathcal P^{(2)}_k)\subseteq \mathcal P^{(1)}_k.
\]

By uniqueness of the canonical filtration induced by defect
(Theorem~\ref{thm:canonical-filtration}), the filtrations agree under these maps,
and the induced maps on associated graded objects are mutually inverse
isomorphisms. Hence $(\mathcal P^{(1)},\circ)$ and $(\mathcal P^{(2)},\circ)$ are equivalent
algebraic Phases in the sense of
Definition~\ref{def:phase-equivalence}. 
\end{proof}

\begin{remark}
An algebraic Phase extracted from fixed admissible phase data may admit multiple
operator realizations, reflecting presentation level degrees of freedom.
Nevertheless, all such realizations yield the same algebraic Phase object up to
equivalence.
\end{remark} 

\subsection*{Defect equivalence principle}

We next address robustness at the level of defect itself.
There are, in principle, many ways to formalize defect, depending on how
non-rigidity is detected from phase interactions.
Algebraic Phase Theory adopts the convention that defect is recorded only
through the canonical filtration and its finite termination behaviour.
Accordingly, distinct defect constructions are identified whenever they 
induce the same filtration and termination depth, since all invariants of
the theory depend only on this induced complexity structure.

\begin{principle}\label{prin:defect-equivalence}
Distinct defect constructions that induce the same canonical filtration and the
same termination depth are \emph{identified} within Algebraic Phase Theory.
\end{principle}

\begin{remark}\label{rem:defect-equivalence}
Principle~\ref{prin:defect-equivalence} is not a theorem about a fixed algebraic
Phase, but a foundational convention of Algebraic Phase Theory. 
The theory is designed to record defect only through the induced intrinsic
complexity filtration and its finite termination behaviour.
Accordingly, two defect formalisms are identified whenever they induce the same
canonical filtration and termination depth.

This identification is harmless: all invariants and statements in APT are
formulated purely in terms of the canonical filtration, its associated graded
object, and termination depth.
Therefore, if two defect packages induce the same filtration, APT cannot
distinguish them by any construction it subsequently performs.
This is formalized in Proposition~\ref{prop:defect-equivalence-wellposed}.
\end{remark}

As usual, we adopt the convention $\mathcal P_{-1}:=0$ when forming associated
graded objects.

\begin{proposition}\label{prop:defect-equivalence-wellposed} 
Let $(\mathcal P,\circ)$ be an algebraic Phase, and suppose we are given two
functorial defect constructions $\Def^{(1)}$ and $\Def^{(2)}$ on $\mathcal P$,
each satisfying Axioms~I-V with respect to the same underlying operations
$\circ$. Let $\{\mathcal P_k^{(1)}\}_{k\ge 0}$ and
$\{\mathcal P_k^{(2)}\}_{k\ge 0}$ be the corresponding canonical filtrations. 

Assume that
\[
\mathcal P_k^{(1)}=\mathcal P_k^{(2)} \quad \text{for all } k\ge 0,
\]
and that the termination depths agree.
Then all filtration derived invariants of Algebraic Phase Theory coincide for
the two defect formalisms. In particular:
\begin{enumerate}[label=\textup{(\alph*)}] 
\item the rigid substructure is the same, namely
$\mathcal P_0^{(1)}=\mathcal P_0^{(2)}$;
\item the defect degree, interpreted as intrinsic complexity, agrees for every
element $T\in\mathcal P$, so that
\[
\deg_{\Def^{(1)}}(T)=\deg_{\Def^{(2)}}(T);
\]
\item the associated graded objects agree canonically,
\[
\operatorname{gr}^{(1)}(\mathcal P)
:=\bigoplus_{k\ge 0}\mathcal P_k^{(1)}/\mathcal P_{k-1}^{(1)}
\cong
\bigoplus_{k\ge 0}\mathcal P_k^{(2)}/\mathcal P_{k-1}^{(2)}
=: \operatorname{gr}^{(2)}(\mathcal P),
\]
with the isomorphism induced by the identity on each filtration piece.
\end{enumerate}
Consequently, from the perspective of Algebraic Phase Theory, which depends only
on these invariants, the defect constructions $\Def^{(1)}$ and $\Def^{(2)}$
encode the same structural information.
\end{proposition}

\begin{proof}
Assume $\mathcal P_k^{(1)}=\mathcal P_k^{(2)}$ for all $k$. 

(a) In Algebraic Phase Theory, the rigid substructure of a Phase is defined to be the
defect zero stratum $\mathcal P_0$, consisting of all elements on which no defect
is detected. These are precisely the elements of minimal complexity, forming the
lowest level of the canonical filtration. 

Since the two defect constructions induce identical filtrations, their level zero
pieces coincide. Hence 
\[
\mathcal P_0^{(1)}=\mathcal P_0^{(2)}.
\]

\textbf{(b)} By Axiom~III, defect induced complexity is measured intrinsically by the canonical
filtration. The defect degree of an element $T\in\mathcal P$ is defined as the
smallest level of the filtration at which $T$ appears: 
\[
\deg_{\Def^{(i)}}(T)=\min\{k:\ T\in \mathcal P_k^{(i)}\}.
\]

Since the two defect constructions produce exactly the same canonical filtration,
each element of $\mathcal P$ appears at the same filtration level in both cases.
Therefore, the defect degree assigned to any element is the same under
$\Def^{(1)}$ and $\Def^{(2)}$.

(c) The associated graded object records the structure of the Phase one complexity
level at a time. It is constructed entirely from the canonical filtration by
taking, at each level, the quotient that isolates what is genuinely new at that
level relative to all lower levels.

Since the two defect constructions induce identical filtrations, the successive
quotients
\[
\mathcal P_k^{(1)}/\mathcal P_{k-1}^{(1)} \quad \text{and} \quad
\mathcal P_k^{(2)}/\mathcal P_{k-1}^{(2)}
\]
agree for every $k$. Consequently, the identity map induces a canonical
isomorphism of associated graded objects,
\[
\operatorname{gr}^{(1)}(\mathcal P)\cong \operatorname{gr}^{(2)}(\mathcal P).
\]

The same reasoning applies to termination. The termination depth is defined as the
least level $N$ for which the filtration stabilizes, meaning $\mathcal P_N=\mathcal P$.
Because the filtrations coincide levelwise, they necessarily stabilize at the
same depth.

Therefore, all invariants of Algebraic Phase Theory derived from defect,
filtration, or termination agree for the two defect constructions. 
\end{proof}

\section{Conclusion}

This paper introduces Algebraic Phase Theory as a framework for extracting
intrinsic algebraic structure from phase interactions and organizing it through
defect, canonical filtration, and finite termination. Starting from minimal
admissible phase data, we showed that defect is not auxiliary but forces a unique,
functorial hierarchy of complexity, and that bounded phase degree imposes an
intrinsic termination of this hierarchy.

The radical quadratic model over $R=\mathbb{F}_2[u]/(u^2)$ serves as the minimal
genuinely nontrivial instance of the theory. In this setting, nilpotence forces
the emergence of nontrivial defect while simultaneously enforcing finite
termination. We proved that quadratic depth is a sharp structural boundary: any
attempt to incorporate genuinely higher degree interactions necessarily
introduces new defect strata and breaks finite control.

From this perspective, Algebraic Phase Theory is not a classification scheme but
a boundary theory. It isolates which algebraic structures are forced by phase
interactions and identifies the precise point beyond which no functorial,
finitely terminating extension is possible. The results of this paper establish
the foundational layer of the theory, on which the subsequent papers in this
series build.

\bibliographystyle{amsplain}
\bibliography{references}

\end{document}